\documentclass{amsart}

\usepackage{amsmath,amssymb}
\usepackage{amsthm}
\usepackage{pinlabel}

\newtheorem{thm}{Theorem}[section]
\newtheorem{conj}[thm]{Conjecture}
\newtheorem{lem}[thm]{Lemma}
\newtheorem{prop}[thm]{Proposition}
\newtheorem{cor}[thm]{Corollary}

\theoremstyle{definition}
\newtheorem{rem}[thm]{Remark}

\newtheorem*{ack}{Acknowledgements}

\numberwithin{equation}{section}

\def\co{\colon\thinspace}

\newcommand{\C}{{\mathbb{C}}}
\newcommand{\N}{{\mathbb{N}}}
\newcommand{\R}{{\mathbb{R}}}
\newcommand{\bp}{{\mathbf{p}}}

\newcommand{\bq}{{\mathbf{q}}}
\newcommand{\Z}{{\mathbb{Z}}}

\DeclareMathOperator{\coker}{\mathrm{coker}}
\DeclareMathOperator{\rank}{\mathrm{rank}}
\DeclareMathOperator{\id}{\mathrm{id}}
\DeclareMathOperator{\rot}{{\tt rot}}
\DeclareMathOperator{\SO}{\mathrm{SO}}
\DeclareMathOperator{\OO}{\mathrm{O}}
\DeclareMathOperator{\UU}{\mathrm{U}}
\DeclareMathOperator{\GL}{\mathrm{GL}}
\DeclareMathOperator{\open}{Open}
\DeclareMathOperator{\tb}{{\tt tb}}

\begin{document}

\title{Diagrams for contact $5$-manifolds}

\author[F. Ding]{Fan Ding}
\address{Department of Mathematics, Peking University,
Beijing 100871, P.~R. China}
\email{dingfan@math.pku.edu.cn}

\author[H. Geiges]{Hansj\"org Geiges}
\address{Mathematisches Institut, Universit\"at zu K\"oln,
Weyertal 86--90, 50931 K\"oln, Germany}
\email{geiges@math.uni-koeln.de}

\author[O. van Koert]{Otto van Koert}
\address{Department of Mathematical Sciences, Seoul National University,
San56-1 Shillim-dong Kwanak-gu, Seoul 151-747, Korea}
\email{okoert@snu.ac.kr}

\begin{abstract}
According to Giroux, contact manifolds can be described as
open books whose pages are Stein manifolds. For $5$-dimensional
contact manifolds the pages are Stein surfaces, which permit
a description via Kirby diagrams. We introduce handle moves on such
diagrams that do not change the corresponding contact manifold.
As an application, we derive classification results for subcritically Stein
fillable contact $5$-manifolds and characterise the standard contact
structure on the $5$-sphere in terms of such fillings. This characterisation
is discussed in the context of the Andrews--Curtis conjecture
concerning presentations of the trivial group. We further illustrate the use
of such diagrams by a covering theorem for
simply connected spin $5$-manifolds and a new existence proof for contact
structures on simply connected $5$-manifolds.
\end{abstract}

\date{}

\maketitle

\section{Introduction}
The aim of this paper is to develop a diagrammatic language for
$5$-dimensional contact manifolds. This is motivated by (but does not depend
on) the deep result of Giroux~\cite{giro02}, which says that any
closed contact manifold (in any odd dimension) admits an open book
decomposition adapted to the contact structure, where the pages are
Stein manifolds and the monodromy is a symplectic diffeomorphism. In the case
of a $5$-dimensional contact manifold, the pages are Stein surfaces.
These permit a description via Kirby diagrams, where the attaching circles
for the $2$-handles are Legendrian knots in the standard contact structure
on the boundary $\#_kS^1\times S^2$ of the $1$-handlebody.
Provided the monodromy is given as a product of Dehn twists
along Lagrangian spheres corresponding to the $2$-handles,
this too can be encoded in the Kirby diagram.

Since Legendrian knots are faithfully represented by their front
projection in the $2$-plane, we obtain a description of $5$-manifolds
in terms of $2$-dimensional diagrams.

The combination of stabilisations of the open book with handle
slides in the Stein page leads to a couple of moves on such diagrams
that do not change the contact $5$-manifold. These moves are introduced
in Section~\ref{section:moves}, after a brief discussion of open books
and their monodromy in Sections \ref{section:book}
and~\ref{section:monodromy}. We then present two simple
applications of these moves to diagrams without $1$-handles.
In Section~\ref{section:S2S3} we describe an integer family of
contact structures on $S^2\times S^3$ and $S^2\,\tilde\times\,S^3$,
the non-trivial $S^3$-bundle over~$S^2$. We also give a diagrammatic
proof of the diffeomorphism
\[ S^2\,\tilde\times\,S^3\,\#\,
S^2\,\tilde\times\,S^3\cong
S^2\times S^3\,\#\,
S^2\,\tilde\times\,S^3\]
and its contact analogue. In Section~\ref{section:no1handles} we
classify contact $5$-manifolds that admit Stein
fillings made up of a single $0$-handle and $2$-handles only.

In Section~\ref{section:subcritical} we turn our attention to general
subcritically Stein fillable contact $5$-manifolds. These can be
described by open books with trivial monodromy, whose
diagrammatic representations are particularly tractable.
We show how to implement the Tietze moves on presentations of
the fundamental group as diagrammatic moves.
As an application, we prove that subcritically
fillable contact $5$-manifolds are classified by their fundamental group,
up to connected sums with $S^2\times S^3$
and $S^2\,\tilde\times\,S^3$ (with their standard contact structures),
see Theorem~\ref{thm:classification}. There is a corresponding
classification of the subcritical Stein fillings
(Corollary~\ref{cor:classification_Stein}).

Roughly speaking, these results can be phrased as saying that
$6$-dimensional compact subcritical Stein manifolds are determined
by topological data. This contrasts sharply with the general situation
for Stein manifolds. In any even dimension $\geq 8$
there are countably many pairwise distinct Stein manifolds
of finite type, all diffeomorphic to Euclidean
space; this was proved by McLean~\cite{mcle09},
building on earlier work of Seidel and Smith~\cite{sesm05}.

In Section~\ref{section:S5} we specialise to subcritical Stein fillings
of the $5$-sphere and give a characterisation of the standard
contact structure on $S^5$ in terms of such fillings. We discuss
the relevance of this result in the context of the
Andrews--Curtis conjecture concerning presentations of the
trivial group.

Finally, in Section~\ref{section:simply} we exhibit
diagrams for some special simply connected $5$-manifolds. These diagrams
are then used to show that every $5$-dimensional simply connected
spin manifold is a double branched cover of the $5$-sphere.
A further application is a new proof that
every simply connected $5$-manifold admits a contact
structure in each homotopy class of almost contact structures.
\section{Open book decompositions}
\label{section:book}
In this section we review the basic aspects of the Giroux
correspondence between contact structures and open books.
We describe three essential operations on open books: Dehn--Seidel twist,
stabilisation, and open book connected sum. In the last part of this
section we recall how to compute the homology of open books.
\subsection{Open books}
\label{section:openbooks_defs}
Recall that a compact Stein manifold is a compact complex manifold
$\Sigma$ admitting a strictly plurisubharmonic function
$f\co\Sigma\rightarrow\R$ that is constant
on the boundary~$\partial\Sigma$ and has no
critical points there. Then the exact $2$-form
$i\partial\overline{\partial}f$ defines a symplectic structure
on~$\Sigma$ compatible with the complex structure. Notice that
compact Stein manifolds are in particular of {\bf finite type},
i.e.\ they have finite handlebody decompositions.

According to a fundamental theorem of Giroux~\cite{giro02},
any cooriented contact structure on a closed manifold is
supported by an open book decomposition of that manifold,
where the pages are compact Stein manifolds and the monodromy
is symplectic. For the purposes of the present article it suffices to
understand how one finds a contact structure adapted to a given open book
decomposition. Here we briefly recall this construction.

Let $(\Sigma,\omega=d\lambda)$ be a compact Stein manifold, and let
$\psi$ be a symplectomorphism of $(\Sigma ,\omega)$ equal to the identity
near~$\partial\Sigma$. By a lemma of Giroux, cf.~\cite[Lemma~7.3.4]{geig08},
we may assume without loss of generality that the
symplectomorphism is exact, that is, $\psi^* \lambda =\lambda+dh$
for some smooth function $h\co\Sigma\rightarrow\R^+$.

The mapping torus
\[ A:=\Sigma \times \R / (x,\varphi)\sim (\psi(x),\varphi-h(x))\]
carries the contact form $\lambda+d\varphi$.
Since $\psi$ equals the identity near~$\partial\Sigma$, we can glue $A$ to 
$B:=\partial \Sigma \times D^2$ along their common boundary
$\partial \Sigma\times S^1$. In terms of polar coordinates $(r,\varphi)$
on~$D^2$, one can define a contact form on $B$ by the ansatz
\[ h_1(r)\, \lambda|_{\partial \Sigma}+h_2(r)\, d\varphi.\]
With the functions $h_1$ and $h_2$ chosen as in
Figure~\ref{fig:functions1}, this will indeed be a contact form
that glues smoothly with $\lambda+d\varphi$ on~$A$, resulting in
a contact form $\alpha$ on $M:=A\cup_{\partial}B$.
This description of the manifold $M$ is called an {\bf open book
decomposition}. The codimension~2
submanifold $\partial\Sigma\times\{ 0\}\subset B\subset M$ is called the
{\bf binding} of the open book. Up to diffeomorphism,
$A$ can be identified with
\[ \Sigma\times [0,2\pi ]/(x,2\pi)\sim (\psi (x),0).\]
In terms of this description, the {\bf pages} of the
open book  are the codimension~$1$ submanifolds
\[ \Sigma\times\{\varphi\}\cup\partial\Sigma\times\{ re^{i\varphi}
\in D^2\co r\in [0,1]\},\]
which are diffeomorphic copies of~$\Sigma$. The map
$\psi$ is called the {\bf monodromy} of the open book.
For more details see \cite[Sections 4.4.2 and~7.3]{geig08}.

\begin{figure}[htp]
\labellist
\small\hair 2pt
\pinlabel $1/2$ [t] at 166 39 
\pinlabel $1/2$ [t] at 488 39 
\pinlabel $h_1$ [r] at 39 192
\pinlabel $h_2$ [r] at 361 192
\pinlabel $r$ [t] at 212 39
\pinlabel $r$ [t] at 534 39
\endlabellist
\centering
\includegraphics[scale=0.5]{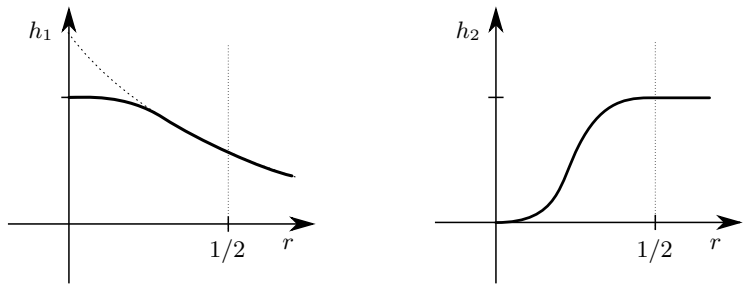}
  \caption{The functions $h_1$ and $h_2$.}
  \label{fig:functions1}
\end{figure}

It is not too difficult to see that the resulting contact manifold
$(M,\ker\alpha )$ is determined, up to contactomorphism,
by $\Sigma$ and~$\psi$. In fact, it is enough to know the
symplectomorphism type of the completion of $\Sigma$
in the sense of~\cite{elgr91}, see~\cite[Proposition~9]{giro02}.
For a compact Stein manifold,
the completion is simply the corresponding open Stein manifold.
We are therefore justified in denoting this contact manifold $(M,\ker\alpha)$
by $\open( \Sigma,\psi)$ and call it a {\bf contact open book}.
When $M$ is $5$-dimensional, the pages
$\Sigma$ of the open book are Stein surfaces, which allow a
description in terms of Kirby diagrams. This description of
contact $5$-manifolds will form the basis of our discussion.
\subsection{Dehn--Seidel twists}
Let $L\cong S^n$ be a Lagrangian sphere in a compact Stein manifold
$(\Sigma,d\lambda)$ of real dimension~$2n$. By the Weinstein neighbourhood
theorem~\cite{wein71} there is a neighbourhood of $L$ symplectomorphic to
the cotangent bundle $T^*S^n$ with its canonical symplectic structure
$d\lambda_{\mathrm{can}}$, which is defined as follows.
Using Cartesian coordinates $(\bq ,\bp) \in \R^{n+1}\times\R^{n+1}$,
we can describe the cotangent bundle $T^*S^{n}\subset\R^{2n+2}$ by the
equations
\[ \bq\cdot\bq=1 \text{ and } \bq\cdot\bp=0;\]
the canonical $1$-form is given by $\lambda_{\mathrm{can}}=\bp\, d\bq$.

For each $k\in\Z$, one can define a so-called $k$-fold Dehn twist
\[ \tau_k\co (T^*S^n,d\lambda_{\mathrm{can}})\longrightarrow
(T^*S^n,d\lambda_{\mathrm{can}}) \]
as follows. First consider the normalised geodesic flow $\sigma_t$ on
$T^*S^n\setminus S^n$ given by
\[ \sigma_t(\bq,\bp)=
\left(
\begin{array}{cc}
\cos t       & |\bp|^{-1} \sin t \\
-|\bp|\sin t & \cos t 
\end{array}
\right)
\binom{\bq}{\bp}.\]
Then set
\[ \tau_k(\bq,\bp)=\sigma_{g_k(|\bp|)}(\bq,\bp),\]
where $r\mapsto g_k(r)$ is a smooth function that interpolates monotonically
between $k\pi$ near $r=0$ and $0$ for large~$r$. For $\bp=0$ we read this
as $\tau_k(\bq,0)=((-1)^k\bq ,0)$.
Then $\tau_k$ is an exact
symplectomorphism of $(T^*S^n,d\lambda_{\mathrm{can}})$, see~\cite{vkni05},
equal to the identity for $|\bp|$ large.
This allows us to regard $\tau_k$ as a symplectomorphism of~$\Sigma$.
Viewed this way, $\tau_k$ is called a $k$-fold {\bf Dehn twist} along
$L\subset\Sigma$.
The map $\tau_1$ is called a right-handed Dehn (or Dehn--Seidel) 
twist~\cite[Section~6]{seid99}, cf.~\cite{seid08}; for $n=1$ this
coincides with the classical notion of a Dehn twist.
\subsection{Stabilisations}
Suppose we are given a contact manifold $\open (\Sigma,\psi)$ and
a properly embedded Lagrangian disc $L\subset\Sigma$ with Legendrian
boundary $\partial L\subset\partial\Sigma$. We can construct a
Stein manifold $\Sigma'$ by attaching an $n$-handle to $\Sigma$
along~$\partial L$. This new Stein (and hence symplectic) manifold contains
a Lagrangian sphere~$L'$, given as the union of $L$ and the core of the
$n$-handle. Let $\tau_{L'}$ be a right-handed Dehn twist along~$L'$.
The contact manifold $\open (\Sigma',\psi\circ\tau_{L'})$ is called
a {\bf right-handed stabilisation} of $\open (\Sigma ,\psi)$
along~$L$.

Giroux has announced the following result.

\begin{prop}[Giroux]
A right-handed stabilisation of $\open (\Sigma ,\psi)$
does not change its contactomorphism type.\qed
\end{prop}

For a detailed proof see~\cite{vkoe}; here is the main idea.
The Legendrian sphere $\partial L\subset\partial\Sigma$ in the
binding of the open book is an isotropic sphere in the ambient contact
manifold $\open(\Sigma, \psi)$. So the attaching of a
handle to each page along $\partial L$ may be seen as a contact surgery
(in the sense of~\cite{elia90,wein91}, cf.~\cite{geig08}) along
this isotropic sphere, where the necessary trivialisation of the
symplectic normal bundle of $\partial L$ in $\open(\Sigma, \psi)$
is provided by the trivialisation of the normal bundle
of the binding $\partial\Sigma$ in the open book.
Performing a right-handed Dehn twist along the Lagrangian sphere $L'$
in the new page $\Sigma'$ may be regarded as a Legendrian
surgery on $L'\subset \open (\Sigma',\psi)$,
see~Theorem~\ref{thm:Dehntwist_legendrian_surgery} below.
This second surgery can be seen to cancel the first, even on the level
of symplectic handle attachments.
\subsection{Book connected sum}
\label{section:sum}
A connected sum operation for open books has been described
in~\cite{mnms06}. Let $\open (\Sigma_i,\psi_i)$, $i=1,2$, be two open
books of dimension $2n+1$. We write $\partial\Sigma_i$ for the
binding. One can form the connected sum of these two
open books by cutting out discs
$D^{2n+1}_i$ embedded in such a way that the pair $(D^{2n+1}_i,
D^{2n+1}_i\cap\partial\Sigma_i)$ is diffeomorphic to
the standard disc pair $(D^{2n+1},D^{2n-1})$.
Then the connected sum
\[ \open (\Sigma_1,\psi_1)\#\open (\Sigma_2,\psi_2)\]
is diffeomorphic to
\[ \open (\Sigma_1\natural\Sigma_2,\psi_1\natural\psi_2),\]
where $\Sigma_1\natural\Sigma_2$ denotes the boundary connected sum of the
pages, and $\psi_1\natural\psi_2$ is the obvious map on this
boundary connected sum that restricts to $\psi_i$ on~$\Sigma_i$.

This construction is compatible with the contact structures on these
open books if $\Sigma_1\natural\Sigma_2$ is interpreted as the
boundary connected sum of Stein or symplectic manifolds
as in~\cite{elia90,wein91}, see~\cite{vkoe05}.
\subsection{Homology of open books}
\label{section:homology_openbook}
The homology of an open book $M=A\cup_{\partial} B$ can easily be computed
in terms of the homology of the page $\Sigma$ and the action of
the monodromy $\psi$ on homology. This will turn out to be
especially useful in the discussion of simply connected $5$-manifolds,
whose diffeomorphism type is determined by homological data.

We only consider this $5$-dimensional case, and we assume
for simplicity that $\Sigma$ is composed of one $0$-handle
and only $2$-handles. With $Q$ denoting the intersection form
on $H_2(\Sigma)$, one finds that $H_1(\partial\Sigma)\cong\coker Q$,
cf.~\cite[pp.~427--8]{bred93}. This is the essential part of the
homology of $B\simeq\partial\Sigma$.

We briefly recall the argument for this statement about
$H_1(\partial\Sigma )$, since we need explicit information about the
homological generators in the proof of Proposition~\ref{prop:M_kN_k}
below. The homology $H_2(\Sigma )$ is isomorphic to~$\Z^m$,
with $m$ denoting the number of $2$-handles, freely generated
by the surfaces obtained by gluing a Seifert surface of
each attaching circle in $S^3=\partial D^4$ with the core disc
of the corresponding $2$-handle. The relative homology
group $H_2(\Sigma ,\partial\Sigma )$ is likewise isomorphic
to~$\Z^m$; here the generators are meridional discs of the
attaching circles whose boundary lies on the boundary of the
$2$-handle. The homology $H_1(\partial\Sigma )$ is
generated by the meridians of the attaching circles, i.e.\
the images of the generators of $H_2(\Sigma ,\partial\Sigma )$
under the boundary homomorphism. In terms of the described generators,
the homomorphism $H_2(\Sigma )\rightarrow H_2(\Sigma ,\partial\Sigma )$
is given by the intersection form~$Q$. The result
$H_1(\partial\Sigma)\cong\coker Q$ now follows from the homology
exact sequence of the pair $(\Sigma ,\partial\Sigma)$.

The homology of the $\Sigma$-bundle $A$ over $S^1$ can be computed using
the Wang sequence~\cite[Lemma~8.4]{miln68}. The relevant part of this
sequence looks as follows.
\[ \ldots\longrightarrow H_3(A) \longrightarrow H_2(\Sigma)
\stackrel{\psi_*-\id}{\longrightarrow} H_2(\Sigma)
\longrightarrow H_2(A) \longrightarrow H_1(\Sigma ) \longrightarrow  
\ldots \]

This information on $A$ and $B$ can be combined to obtain the homology of
$M$ via the Mayer--Vietoris sequence of the decomposition
$M=A\cup_{\partial} B$.
\section{Monodromy}
\label{section:monodromy}
The symplectomorphism group of a given Stein manifold $\Sigma$ is not known,
in general. Therefore we restrict our attention to symplectomorphisms
that can be written as compositions of Dehn twists along Lagrangian
spheres $L\subset\Sigma$.
\subsection{Action of Dehn twists on homology}
\label{section:action}
Let $\Sigma$ be a symplectic manifold of dimension~$2n$ and $L\subset\Sigma$
a Lagrangian sphere. Write $[L]\in H_n(\Sigma)$ for the homology class
represented by~$L$. As before, we denote the intersection form
on $H_n(\Sigma)$ by~$Q$. With this notation, the homomorphism
on homology induced by a right-handed Dehn twist $\tau_L$
along $L$ is given by
\[
\begin{array}{rcl}
(\tau_L)_*\co H_n(\Sigma) & \longrightarrow & H_n(\Sigma)\\
u                         & \longmapsto     & u+(-1)^{1+n(n+1)/2}
                                                Q([L],u)\cdot [L].
\end{array}\]
This can be seen as follows. First suppose $u$ is a homology
class represented by a closed, oriented submanifold of~$\Sigma$. This
submanifold can be isotoped to intersect $L$ transversely in
a finite number of points. Think of a neighbourhood of $L$ in
$\Sigma$ as $T^*L$; the submanifold representing $u$ may then be
assumed to intersect $T^*L$ in a finite number of fibres. On the level
of homology, a Dehn twist along $L$ adds $\pm [L]$ to $u$ for each
of these intersections. The dependence of the sign in $\pm [L]$
on the sign of the transverse intersection can be
computed explicitly in the local model; we leave this to the
reader.

The same argument applies for any singular cycle~$u$, as can be verified by
using the intersection theory developed in~\cite[{\S}73]{seth34} for pairs
of singular chains of complementary dimensions in a given manifold.
\subsection{Monodromy and Stein filling}
\label{section:monodromy_filling}
For the most part we shall be concerned with open books
that have trivial monodromy. The following proposition tells
us that this is a reasonably interesting class of manifolds
to consider. Recall that a contact manifold $(M,\xi )$ is said to be
{\bf Stein fillable} if it arises as the boundary of a compact Stein
manifold, with $\xi$ given as the complex tangencies on the boundary.
As is well known, a Stein manifold of real dimension $2n+2$ has a
handle decomposition with handles up to index $n+1$ only. The filling
is called {\bf subcritical} if there are no handles of index $n+1$.

\begin{prop}
\label{prop:subcritical_monodromy}
A contact structure $\xi$ on a closed manifold $M$ is supported
by an open book with trivial monodromy if and only if $(M,\xi )$
is subcritically Stein fillable.
\end{prop}

\begin{proof}
Suppose $(M,\xi)=\open (\Sigma,\id)$. Then
\[ M=\Sigma\times S^1\cup_{\partial}\partial\Sigma\times D^2=
\partial (\Sigma\times D^2).\]
If $J$ is the complex structure on $\Sigma$ and $f\co\Sigma\rightarrow\R$
a strictly plurisubharmonic function, then $\Sigma\times D^2\ni (p,x,y)
\mapsto f(p)+x^2+y^2$ defines a strictly plurisubharmonic function
on $\Sigma\times D^2$ with respect to the obvious complex structure $(J,i)$,
and this function has no critical points of maximal index.

The converse is due to Cieliebak~\cite{ciel02}, who showed
that any subcritical Stein manifold is equivalent to a product
$\Sigma\times\C$.
\end{proof}

According to a result of Loi--Piergallini and Giroux~\cite{giro02},
cf.~\cite{geig}, the $3$-dimensional Stein fillable contact manifolds
are characterised as those contact manifolds that admit a supporting
open book whose monodromy is a composition of right-handed Dehn twists.
In higher dimensions, this condition on the monodromy is still
sufficient for Stein fillability.

\begin{prop}
If a contact manifold is supported by an open book
whose monodromy is a composition of right-handed Dehn twists
along Lagrangian spheres on the pages, then it
is Stein fillable.\qed
\end{prop}

This appears to be a folklore theorem; a proof can be found
in~\cite{vkoe}, see also~\cite{akar10}.
The main idea is that a Lagrangian sphere on a page
can be made Legendrian with respect to the contact structure on the open
book; a right-handed Dehn twist then corresponds to a Legendrian
surgery, which preserves Stein fillability.

In \cite{bovk10} left-handed Dehn twists are used to construct
algebraically overtwisted contact manifolds, i.e.\ contact manifolds
with vanishing contact homology, which are conjecturally not strongly
fillable (and hence not Stein fillable).
\section{Diagrams for $5$-manifolds}
We now want to give diagrammatic representations of $5$-dimensional
contact open books. Here the page is a Stein surface, which can
be described by a Kirby diagram. According to a result
of Eliashberg~\cite{elia90}, worked out further by Gompf~\cite{gomp98},
any compact Stein surface can be obtained from the $4$-disc
in $\C^2$ by attaching a finite number of $1$- and $2$-handles,
where each $2$-handle is attached along a Legendrian knot
(in the standard contact structure on the relevant boundary $3$-manifold),
with framing $-1$ relative to the contact framing of that
Legendrian knot. Conversely (this is the easier part), this recipe always
produces a Stein surface.

The boundary $3$-manifold obtained by attaching $k$ $1$-handles
to the $4$-disc is the connected sum $\#_k(S^1\times S^2)$
with its unique tight contact structure; the attaching circles for the
$2$-handles form a Legendrian link in this contact manifold.
In other words, the Kirby diagram of a Stein surface consists
of a finite collection of pairs of attaching balls
for the $1$-handles in $\R^3\subset S^3$ with its standard contact structure
$\xi_{\mathrm{st}}=\ker (dz+x\, dy)$,
and the front projection to the $yz$-plane of a Legendrian link.
For more information see \cite{gomp98} and~\cite{gost99}.

In order to obtain a representation of the contact manifold
$\open (\Sigma ,\psi)$, one also needs to encode the monodromy
$\psi$ in the diagram. This is possible in special cases. For instance,
if one of the Legendrian knots in the Kirby diagram bounds in
an obvious way a Lagrangian disc in the Stein surface~$\Sigma$,
the union of this disc with the core disc of the $2$-handle
is a Lagrangian sphere, and we can speak of a Dehn twist along this sphere.
Beware that different such Dehn twists will not, in general,
commute with each other. In cases
where the order of the Dehn twists is inessential, we simply
write the relevant Dehn twist next to the Legendrian knot
representing the corresponding Lagrangian sphere. In particular,
stabilisations can be so encoded, provided one understands the
monodromy of the given diagram.

By Section~\ref{section:sum}, the diagram for the connected
sum of two contact open books is simply given by drawing the
attaching balls and Legendrian attaching circles in a single diagram,
separated by a hyperplane.
\subsection{Handle moves}
\label{section:moves}
We illustrate the use of (de-)stabilisations in the following two
handle moves on a diagram of $\open (\Sigma,\psi)$.
These moves do not change the contactomorphism type of
$\open (\Sigma,\psi)$.
\subsubsection{Move I}
Assume we are given a Kirby diagram for $\Sigma$ that includes
a Legendrian knot $K$ with the property that the monodromy $\psi$
equals the identity on the handle attached along~$K$.
Add a standard Legendrian unknot $K_0$
with Thurston--Bennequin invariant $\tb (K_0)=-1$ to the diagram,
unlinked with the given Legendrian link and not passing over
$1$-handles. This $K_0$ bounds an obvious Lagrangian disc,
so a right-handed Dehn twist $\tau$ along the corresponding
$2$-sphere (see the following remark)
amounts to a stabilisation of $\open (\Sigma,\psi)$.

\begin{rem}
It is not accidental that the knot $K_0$ we use for the
stabilisation is chosen to have $\tb (K_0)=-1$. Write $L'$ for
the Lagrangian $2$-sphere obtained by gluing the Lagrangian disc $L$
bounded by $K_0$ to the core disc of the handle attached along~$K_0$.
A neighbourhood of $L'$ looks like $T^*S^2$, so $L'$ has
self-intersection~$-2$. If we push the core disc along its boundary
in the direction of the surgery framing, we obtain a disjoint disc.
Hence, pushing $L$ in that direction along its boundary leads to
a disc having intersection $-2$ with $L$. If $L$ is topologically
isotopic to a disc in $\partial\Sigma$ (as in the case of~$K_0$),
this means that the linking number of $K_0$ with its push-off
in the direction of the surgery framing equals~$-2$. Since the
surgery framing is obtained from the contact framing by adding
a negative twist, we conclude $\tb (K_0)=-1$.

Beware that $\tb (K_0)$ may take other values, with $K_0$ regarded
as a  knot in $(S^3,\xi_{\mathrm{st}})$, if the Lagrangian
disc bounded by $K_0$ goes over $2$-handles.
\end{rem}

Move I is now performed in three steps, see Figure~\ref{fig:move1}:
\begin{itemize}
\item[(i)] Stabilise $\open (\Sigma,\psi)$ by adding
$(K_0,\tau)$ to the diagram as described.
\item[(ii)] Perform a handle slide of $K$ over $K_0$; this is possible via
a Legendrian isotopy, see~\cite[Proposition~1]{dige09}.
\item[(iii)] Destabilise by removing $(K_0,\tau)$ from the picture.
\end{itemize}

\begin{figure}[htp]
\labellist
\small\hair 2pt
\pinlabel {\scriptsize id} [tr] at 27 168
\pinlabel {\scriptsize id} [tr] at 27 66
\pinlabel {\scriptsize id} [tr] at 196 168
\pinlabel {\scriptsize id} [tr] at 190 74
\pinlabel {\scriptsize stabilise} [b] at 101 155
\pinlabel {\scriptsize handle slide} [br] at 145 106
\pinlabel {\scriptsize destabilise} [b] at 145 15
\pinlabel $\tau$ [t] at 280 166
\pinlabel $\tau$ [t] at 110 70
\endlabellist
\centering
\includegraphics[scale=0.9]{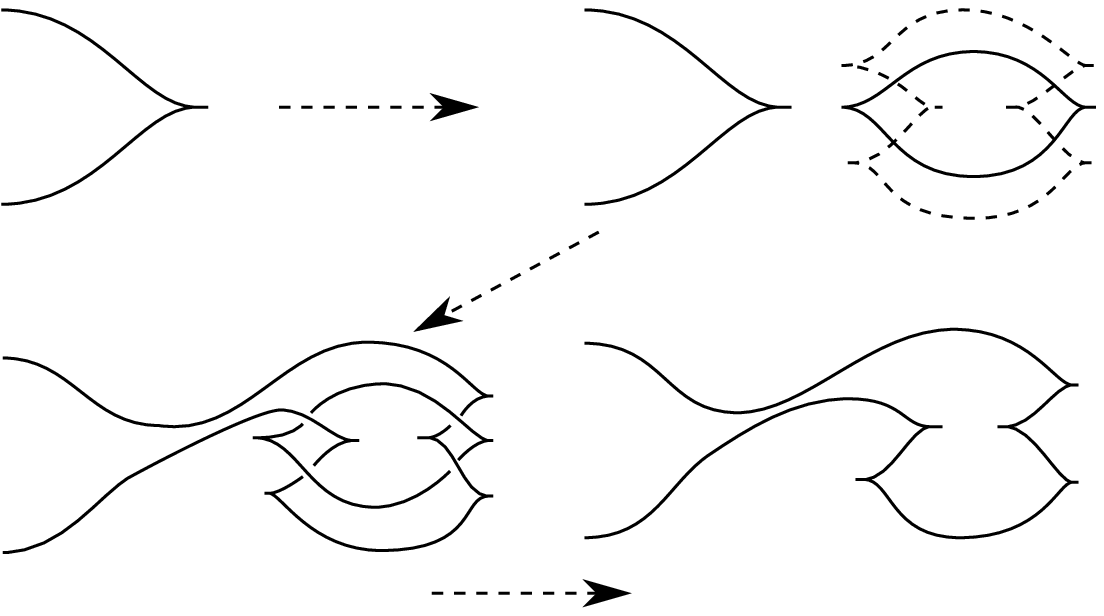}
  \caption{Move~I.}
  \label{fig:move1}
\end{figure}

\begin{rem}
\label{rem:extension}
A Legendrian isotopy of a Legendrian submanifold in a contact manifold
$(M,\xi =\ker\alpha)$ extends to a contact isotopy $\phi_t$ of $(M,\xi)$,
cf.~\cite[Theorem~2.6.2]{geig08}, and thence to an
$\R$-invariant symplectic isotopy $\Phi_t$
of the symplectisation $(\R\times M, d(e^s\alpha))$. Write
$\phi_t^*\alpha=e^{h_t}\alpha$ with a smooth function $h_t\co M\rightarrow\R$
and set $\Phi_t(s,x)=(s-h_t(x),\phi_t(x))$.
Then $\Phi_t^*(e^s\alpha)=e^s\alpha$, which implies that $\Phi_t$ is a
Hamiltonian isotopy. By cutting off the corresponding Hamiltonian
function, this isotopy may be assumed to coincide with
$\Phi_t$ on $\{ r\}\times M$ and to be stationary outside
$(0,R)\times M$ for $R>r>0$ sufficiently large.

Hence, if $(M,\xi)$ has a strong symplectic filling $(W,\omega)$, 
this cut off symplectic isotopy extends to an isotopy
of the symplectic completion of $(W,\omega)$ in the
sense of~\cite{elgr91}. So the time-$1$ map $\Phi_1$ may be regarded as
a symplectomorphism of the symplectic completion of $(W,\omega )$.
Alternatively, we may view $\Phi_1$ as a symplectomorphism of the filling
$(W,\omega)\cup \bigl([0,r]\times M, d(e^s\alpha)\bigr)$ of $(M,\xi)$
onto its image, which on the boundary induces the contactomorphism~$\phi_1$.

It will be in this sense that we interpret step (ii) of move~I
as a symplectomorphism of the filling.
\end{rem}

The effect of this move~I is to replace $K$ by its double
Legendrian stabilisation $S_+S_-K$, i.e.\
a Legendrian knot which has two additional zigzags (one up, one down).
Notice that the contact framing of $S_+S_-K$ differs from that
of $K$ by two negative (i.e.\ left-handed) twists; if $K$ is
homologically trivial this means $\tb (S_+S_-K)=\tb (K)-2$.
The rotation number $\rot$ does not change under this move.

One of the potential uses of move~I is the following.
As shown by Fuchs and Tabachnikov~\cite[Theorem~4.4]{futa97},
cf.~\cite{dige},
any topological isotopy of Legendrian knots can be turned into a Legendrian
isotopy of suitable Legendrian stabilisations.
Thus, after a repeated application
of move~I, two topologically isotopic Legendrian
knots with the same rotation number will become Legendrian isotopic.
\subsubsection{Move II}
Our second move gives us even greater flexibility, for it allows us
to change crossings of Legendrian knots $K,K'$ (or a self-crossing of~$K$)
in the Kirby diagram for~$\Sigma$, at the cost of replacing $K$
by its fourfold Legendrian stabilisation $S_+^2S_-^2K$. The move is
a combination of move~I with a further (de-)stabilisation and
handle slide; a pictorial description is given in Figure~\ref{fig:move2}.

\begin{figure}[htp]
\labellist
\small\hair 2pt
\pinlabel {\scriptsize id} [tr] at 30 208
\pinlabel {\scriptsize id} [tr] at 24 72
\pinlabel {\scriptsize id} [tr] at 244 208
\pinlabel {\scriptsize id} [tr] at 244 72
\pinlabel {\scriptsize move I and stabilise} [b] at 169 191
\pinlabel {\scriptsize destabilise} [b] at 204 40
\pinlabel {\scriptsize handle slide} [br] at 218 125
\pinlabel $\tau$ [t] at 330 208
\pinlabel $\tau$ [t] at 112 69
\endlabellist
\centering
\includegraphics[scale=0.9]{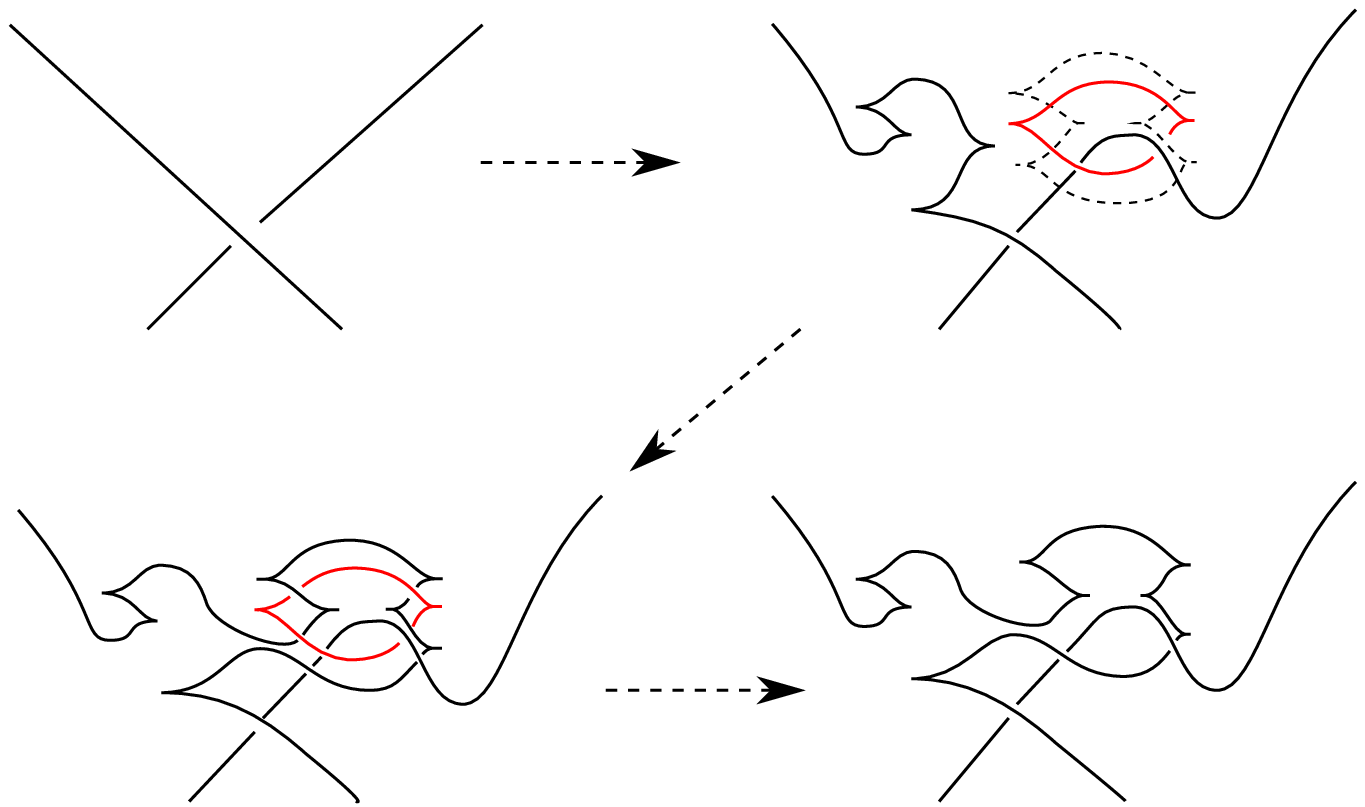}
  \caption{Move II.}
  \label{fig:move2}
\end{figure}

The first two steps in Figure~\ref{fig:move2} can be performed even
if the monodromy along $K$ is not the identity, but the
destabilisation may not be possible. Move~II without the final
destabilisation still allows us to change a given Kirby diagram
into one where the Legendrian link is topologically a link of unknots.
\subsection{Diagrams for $S^2\times S^3$ and $S^2\,\tilde\times\,S^3$}
\label{section:S2S3}
We now use moves I and~II from the preceding section to give a characterisation
of the contact manifolds $\open (\Sigma ,\id)$ where $\Sigma$
has a Kirby diagram consisting of a single knot~$K$.
 
Since $\pi_1(\SO(4))=\Z_2$, there are exactly two $S^3$-bundles over
$S^2$ up to bundle isomorphism, the trivial and the non-trivial one.
Their total spaces are non-diffeomorphic and are denoted by
$S^2\times S^3$ and $S^2\,\tilde\times\,S^3$, respectively.

If $M$ is a closed simply connected $5$-manifold with $H_2(M;\Z)\cong \Z$,
then by Barden's classification~\cite{bard65} one has
\[ M\cong
\left\{
\begin{array}{ll}
S^2\times S^3          & \text{ if }w_2(M)=0, \\
S^2\,\tilde\times\,S^3 & \text{ if }w_2(M)\neq 0. \\
\end{array}
\right. \]

\begin{prop}
\label{prop:S2S3}
Let $K$ be an oriented Legendrian knot in $(S^3,\xi_{\mathrm{st}})$.
Then the contact manifold $(M,\xi)=\open (\Sigma,\id)$, with $\Sigma$
the Stein surface given by the Kirby diagram consisting of $K$ only,
is one of the following:
\begin{itemize}
\item{} $(S^2\times S^3,\xi_{\lvert\rot(K)\rvert})$ if $\,\rot(K)$ is even,
\item{} $(S^2\,\tilde\times\,S^3,\xi_{\lvert\rot(K)\rvert})$
        if $\,\rot(K)$ is odd.
\end{itemize}  
As the notation suggests, the contact structure $\xi_{\lvert\rot(K)\rvert}$
depends, up to diffeomorphism, only on $\lvert\rot(K)\rvert$.
\end{prop}

\begin{proof}
Using move~II we can untie any Legendrian knot without changing its
rotation number, so we may assume that $K$ is an unknot.
Then $\Sigma$ is a $2$-disc bundle over~$S^2$. The zero section $S^2$ of this
bundle, oriented in such a way that the disc $D^2\subset S^2$
bounded by $K$ in $S^3$ is oriented consistently with~$K$,
represents the positive generator of $H_2(\Sigma)$.
Since the monodromy is the identity, we have
$M\cong \partial (\Sigma \times D^2)$.
This means that $M$ is the boundary of a $4$-disc bundle over~$S^2$,
i.e.\ an $S^3$-bundle.

Next we determine the first Chern class of~$\xi$.
The cohomology exact sequence for the pair
$(\Sigma \times D^2,\partial (\Sigma \times D^2))$ shows that the
inclusion $i\co\partial (\Sigma\times D^2)\rightarrow\Sigma\times D^2$
induces an isomorphism $i^*\co H^2(\Sigma \times D^2)\rightarrow
H^2(\partial(\Sigma \times D^2))$.
The contact structure $\xi$ is given by the complex tangencies
of $\partial(\Sigma \times D^2)$ (after the smoothing of corners),
and the complementary complex line bundle in $T(\Sigma\times D^2)|_{\partial
(\Sigma\times D^2)}$ is trivial, so we have $c_1(\xi)=i^*c_1
(\Sigma\times D^2)$. The class $c_1(\Sigma\times D^2)$ can be
naturally identified with $c_1(\Sigma)$, which by
\cite[Proposition~2.3]{gomp98} equals $\rot (K)\, h$,
where $h$ is the positive generator of $H^2(\Sigma)$. Thus, under the
mentioned identifications we have $c_1(\xi)=\rot (K)\, h$.

The second Stiefel--Whitney class $w_2(M)$ equals the mod~2
reduction of $c_1(\xi)$. So the diffeomorphism type of $M$
is as claimed.

It remains to show that the contact structure is determined by the
absolute value of its first Chern class. Thus, let $K_1$ and $K_2$ be two
oriented Legendrian unknots with $\rot (K_1)=\pm\rot (K_2)$.
The orientation of these knots has no bearing on the
resulting contact manifold, only on the designation of one
of the generators of $H_2$ as the positive one. So we may orient
$K_1$ and $K_2$ such that $\rot (K_1)=\rot (K_2)$.
The parity condition $\tb (K_i)+\rot (K_i)\equiv 1$ mod~$2$,
cf.~\cite{geig08}, allows us to achieve in addition that
$\tb (K_1)=\tb (K_2)$ after a repeated application of move~I to one of the
two knots. From the classification
of Legendrian unknots by Eliashberg and Fraser~\cite{elfr09}
it follows that $K_1$ and $K_2$ are then Legendrian isotopic,
so the corresponding $5$-dimensional contact manifolds are contactomorphic.
\end{proof}

\begin{rem}
Both $S^2\times S^3$ and $S^2\,\tilde\times\,S^3$ admit
orientation-preserving diffeomorphisms that act as minus the
identity on~$H^2$ \cite[Theorem~2.2]{bard65}, and
orientation-reversing diffeomorphisms that act as the
identity on $H^2$ \cite[p.~399]{geig08}.
This explains why the orientation of $M$ and the sign
of the rotation number are irrelevant for the diffeomorphism
classification.
\end{rem}

From Barden's classification it follows that
$S^2\,\tilde\times\,S^3 \# S^2\,\tilde\times\,S^3$ is diffeomorphic to
$ S^2 \times S^3 \# S^2\,\tilde\times\,S^3$.
This can also be seen by diagram moves, which proves a little more.

\begin{prop}
\label{prop:connect_sum_S2S3}
With the notation from the preceding proposition, we have a
contactomorphism
\[ (S^2\,\tilde\times\,S^3,\xi_{2m+1})\# (S^2\,\tilde\times\,S^3,\xi_1 )
\cong (S^2 \times S^3,\xi_{2n}) \# (S^2\,\tilde\times\,S^3,\xi_1)\]
for any $m,n\in\N_0$.
\end{prop}

\begin{proof}
The proof consists of sliding the $2$-handle corresponding
to $\xi_{2m+1}$ over the one corresponding to~$\xi_1$;
Figure~\ref{fig:connected_sum_S2S3} shows this for $m=0$, $n=1$.
From (i) to (ii) one performs a handle slide; from (ii) one
gets to (iii) by using move~II to change crossings and 
the reverse of move~I to perform double Legendrian destabilisations.
Observe that the rotation numbers add under a handle slide.
If we perform a handle subtraction instead of a handle
addition (i.e.\ a handle slide with the orientation of
the shark on the right-hand side reversed), we obtain $n=0$.
\end{proof}

\begin{figure}[htp]
\labellist
\small\hair 2pt
\pinlabel {\small (i)} [b] at 14 99
\pinlabel {\small (ii)} [b] at 310 99
\pinlabel {\small (iii)} [b] at 618 99
\endlabellist
\centering
\includegraphics[scale=0.38]{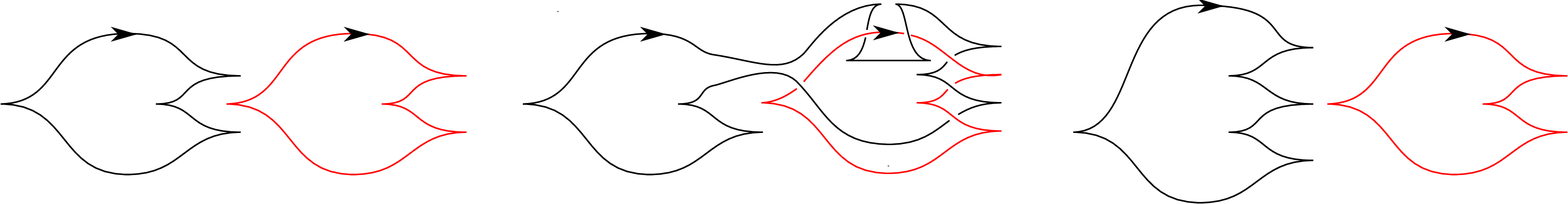}
  \caption{The connected sum $S^2\,\tilde\times\,S^3\# S^2\,
           \tilde\times\,S^3$.}
  \label{fig:connected_sum_S2S3}
\end{figure}
\subsection{Subcritical fillings without $1$-handles}
\label{section:no1handles}
We are now in a position to classify $5$-dimensional contact
manifolds that admit subcritical Stein fillings without
$1$-handles. The first Chern class $c_1(\xi)$ of a contact structure
on a simply connected $5$-manifold determines the homotopy class
of the underlying reduction of the structure group
to $\UU(2)\times 1$. When there exists a
subcritical Stein filling, it actually determines the
contact structure.

\begin{thm}
\label{thm:c_1_determines_5manifold}
Let $(M_i,\xi_i)$, $i=1,2$, be two simply connected contact $5$-manifolds
that admit subcritical Stein fillings without $1$-handles.
If there is an isomorphism $\phi\co H^2(M_1)\rightarrow H^2(M_2)$
such that $\phi(c_1(\xi_1))=c_1(\xi_2)$, then $(M_1,\xi_1)$ and
$(M_2,\xi_2)$ are contactomorphic. 
\end{thm}

\begin{proof}
The assumption on the existence of subcritical Stein fillings implies that
the $(M_i,\xi_i)$ can be realised as contact open books whose
monodromy is the identity. With the help of moves I and~II,
the corresponding diagrams (which, by assumption, contain no
$1$-handles) can be turned into
a collection of unlinked Legendrian unknots. If one so wishes,
one can arrange the rotation number of at most one of the knots
to be odd by the argument in the preceding proposition. In particular,
we see that the $M_i$ are diffeomorphic to a connected sum
of copies of $S^2\times S^3$, possibly with one additional
summand $S^2\,\tilde\times\,S^3$.

Fixing an ordering of the unknots in the respective diagram, and an
orientation for each unknot, amounts to fixing an identification
of $H^2(M_i)$ with $\Z^k$, where $k\in\N$ is the number of
unknots in either diagram. Then $\phi$ may be regarded as an element
of $\GL (k,\Z)$. By elementary row and column operations (over the
integers), this matrix can be converted to the identity matrix.
So it suffices to show that these elementary operations can
be effected by a change in the respective diagram. The
condition $\phi(c_1(\xi_1))=c_1(\xi_2)$ implies
(by Proposition~\ref{prop:S2S3} and its proof) that
between the two new diagrams there is a bijective pairing
of Legendrian unknots with the same rotation number,
which amounts to the claimed contactomorphism.

Here are the diagrammatic realisations of the elementary row
and column operations, the former being changes in the diagram
for~$M_2$, the latter in that of~$M_1$.

1. Add a row/column to another one: this amounts to a handle slide,
see Figure~\ref{fig:cohom_map1}.

\begin{figure}[htp]
\labellist
\small\hair 2pt
\pinlabel $e_i$ [bl] at 23 234
\pinlabel $e_i$ [bl] at 23 153
\pinlabel $e_i+e_j$ [bl] at 23 64
\pinlabel $e_j$ [bl] at 168 190
\pinlabel $e_j$ [bl] at 168 107
\pinlabel $e_j$ [bl] at 168 19
\endlabellist
\centering
\includegraphics[scale=0.8]{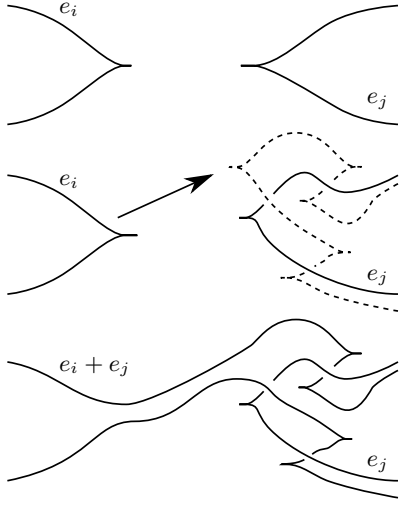}
  \caption{Sliding the handles to get the desired map on cohomology.}
  \label{fig:cohom_map1}
\end{figure}

2. Multiply a row/column by~$-1$:
a change of sign of one of the generators of $\Z^k$ simply amounts to
changing the orientation on the corresponding unknot. This reverses the
sign of the rotation number of that unknot.
\end{proof}

\begin{rem}
M.-L.~Yau~\cite{yau04} has shown that the contact homology
of a subcritically Stein fillable contact manifold $(M,\xi)$
is determined by the homology of the manifold, provided
$c_1(\xi)$ evaluates to zero on $\pi_2(M)$ (i.e.\
homology classes represented by maps $S^2\rightarrow M$).
In particular, if we write the subcritical Stein filling of
a $5$-dimensional contact manifold as $\Sigma\times D^2$, then
$HC_2(M,\xi )\cong H_2(\Sigma)\cong H_2(M)$.
In view of this result, the theorem above says that
contact homology is a complete invariant for $5$-dimensional contact
manifolds $(M,\xi )$ that admit a subcritical
Stein filling without $1$-handles and with $c_1(\xi)=0$.
\end{rem}
\section{Diagrams for subcritically fillable contact $5$-manifolds}
\label{section:subcritical}
In this section we prove a result that goes some way towards classifying
subcritically fillable contact $5$-manifolds and their Stein fillings.
By realising Tietze moves on group presentations via handle
moves in Kirby diagrams, we show that subcritically fillable
contact $5$-manifolds are determined, up to a connected sum with
copies of $S^2\times S^3$ and $S^2\,\tilde\times\,S^3$, by their
fundamental group; see Theorem~\ref{thm:classification}
for the precise formulation. The corresponding classification
of $6$-dimensional subcritical Stein manifolds up to
symplectomorphism is formulated in Corollary~\ref{cor:classification_Stein}.
\subsection{Tietze moves and Legendrian isotopies}
Let $\langle g_1,\ldots,g_k |r_1,\ldots,r_l \rangle$ be a finite
presentation of a group~$G$. We can then realise this group $G$ as the
fundamental group of a Stein surface by associating a $1$-handle with
each of the generators $g_1,\ldots,g_k$, and an oriented
Legendrian attaching circle
with each of the relations $r_1,\ldots,r_l$. Our conventions for
the translation from a group presentation to a Kirby diagram
in standard form as in \cite{gomp98} are as follows.

\begin{itemize}
\item[(i)] The $1$-handles are represented by horizontal pairs
of attaching balls.
\item[(ii)] A Legendrian curve going over the $1$-handle corresponding to a
generator $g$ in the way shown in Figure~\ref{figure:one-handle1}
is read as the letter $g$ in the word represented by that curve.
\item[(iii)] The relations, which are words in the generators, are translated
into a curve by reading the word from left to right.
\end{itemize}

\begin{figure}[htp]
\centering
\includegraphics[scale=0.5]{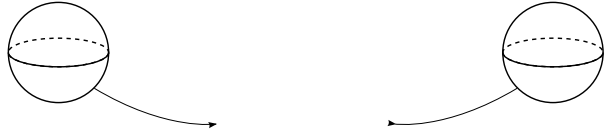}
  \caption{Attaching circle going over a $1$-handle.}
  \label{figure:one-handle1}
\end{figure}
\subsubsection{Equivalent diagrams}
\label{section:trivial_monodromy_equivalence}
Since we are dealing with subcritical Stein fillings,
we may restrict our attention to open books having trivial monodromy.
So we are free to use the moves introduced in Section~\ref{section:moves};
these do not change the contact manifold or its filling.

\begin{itemize}
\item[(i)] Move~I: replace a Legendrian attaching knot $K$ by
its double Legendrian stabilisation $S_+S_-K$.
\item[(ii)] Move~II: change crossings of the Legendrian attaching circles
at the price of adding Legendrian stabilisations.
\end{itemize}

In particular, the crossings of the Legendrian attaching circles
are ultimately of no importance.

In order to realise all Tietze moves on group presentations ---
these moves will be described presently --- we need to
allow one further change in the diagram:

\begin{itemize}
\item[(iii)] Add an unlinked standard Legendrian unknot $K_0$ with
$\tb (K_0)=-1$ to the diagram.
\end{itemize}

On the level of contact $5$-manifolds, this third move
corresponds to taking the connected sum with $(S^2\times S^3,\xi_0)$,
see Proposition~\ref{prop:S2S3}.

We remark that the particular points where the attaching circles go over
the $1$-handles are irrelevant. This is illustrated in
Figure~\ref{figure:changeorder} (without loss of generality, strand 2
is assumed to be Legendrian stabilised).

\begin{figure}[htp]
\labellist
\small\hair 2pt
\pinlabel $1$ [l] at 211 324
\pinlabel $1$ [l] at 211 189
\pinlabel $1$ [l] at 211 57
\pinlabel $2$ [l] at 211 283
\pinlabel $2$ [l] at 211 149
\pinlabel $2$ [l] at 211 22
\endlabellist
\centering
\includegraphics[scale=0.5]{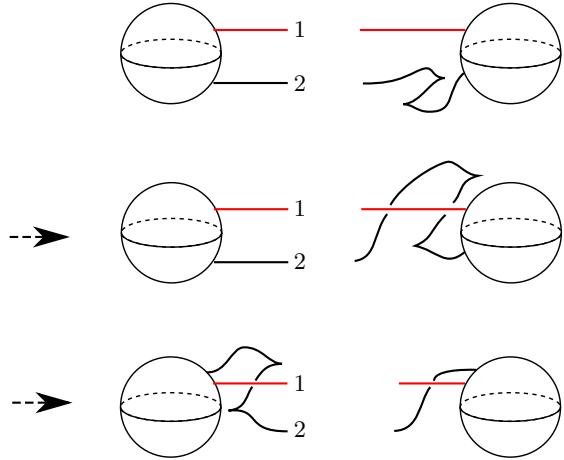}
  \caption{Changing the order of attaching points.}
  \label{figure:changeorder}
\end{figure}
\subsubsection{Tietze moves}
According to the Tietze theorem~\cite[pp.~43--4]{crfo77},
any two finite presentations of a given group are related by a sequence
of the following two moves.

\vspace{1mm}

{\bf T 1.}
Add or remove a relation
$s$ that is a so-called consequence of the other relations:
\[ \langle g_1,\ldots ,g_k | r_1,\ldots, r_l \rangle \leftrightsquigarrow
\langle g_1,\ldots ,g_k | r_1,\ldots, r_l,s \rangle\]
For the relation $s$ to be  a {\bf consequence} of the relations
$r_1,\ldots ,r_l$ means that $s$ (which is an element
of the free group generated by $g_1,\ldots ,g_k$) is contained in
every normal subgroup that contains $r_1,\ldots ,r_l$.

This move, read from left to right, can be written as
a sequence of the following submoves.

\begin{itemize}
\item[(i)] Double a relation or add an inverse of a relation:
\[ \langle g_1,\ldots ,g_k | r_1,\ldots, r_l \rangle \rightsquigarrow
\langle g_1,\ldots ,g_k | r_1,r_1^{\pm 1},\ldots, r_l \rangle\]
\item[(ii)] Conjugate a relation by a generator:
\[r \rightsquigarrow g r g^{-1} \text{ or } g^{-1}rg\]
\item[(iii)] Replace one relation by its product with another relation:
\[ r_1,r_2 \rightsquigarrow r_1,r_1r_2 \text{ or }
r_1,r_2r_1\]
\end{itemize}

\vspace{1mm}

{\bf T 2.}
Add or remove a generator $g$ and a relation $g=w$
expressing $g$ as a word $w$ in the other generators
and their inverses:
\[ \langle g_1,\ldots,g_k|r_1,\ldots ,r_l \rangle 
\leftrightsquigarrow 
\langle g_1,\ldots ,g_k,g | r_1,\ldots, r_l,gw^{-1} \rangle\]
\subsubsection{Realising the Tietze moves}
\label{section:realisingT}
We now want to show that these moves correspond to Legendrian isotopies of
the attaching circles and handle cancellations, modulo the equivalences
described in Section~\ref{section:trivial_monodromy_equivalence}.

\vspace{1mm}

{\bf T 1.} (i)
A relation represented by a Legendrian attaching circle $K$
can be doubled by adding the Legendrian push-off of $K$ to the diagram.
Giving this push-off the opposite orientation amounts to
adding the inverse relation.
By \cite[Proposition~2]{dige09}, this Legendrian push-off is
Legendrian isotopic to a meridian of~$K$. With the help of crossing
changes we can disentangle this meridian from the rest of the diagram.
In other words, we can double a relation by first adding
a standard Legendrian unknot to the diagram, i.e.\ by taking
a connected sum with $(S^2\times S^3,\xi_0)$, and then turning this
into a push-off of $K$ with the help of moves I (and
its inverse) and~II.

\vspace{1mm}

{\bf T 1.} (ii)
The Legendrian knot representing $grg^{-1}$ as shown in
Figure~\ref{figure:tietze1ii} is Legendrian isotopic to that
representing~$r$: first move the cusp formed by the strands $1$ and
$4$ over the $1$-handle~$g$, then perform a first Reidemeister move.

\begin{figure}[htp]
\labellist
\small\hair 2pt
\pinlabel $g$ [r] at 21 395
\pinlabel $1$ [tl] at 359 197
\pinlabel $2$ [tr] at 155 197
\pinlabel $3$ [bl] at 134 271
\pinlabel $4$ [br] at 366 271
\pinlabel $r$ [tl] at 279 18
\pinlabel $grg^{-1}$ [tl] at 400 307
\endlabellist
\centering
\includegraphics[scale=0.45]{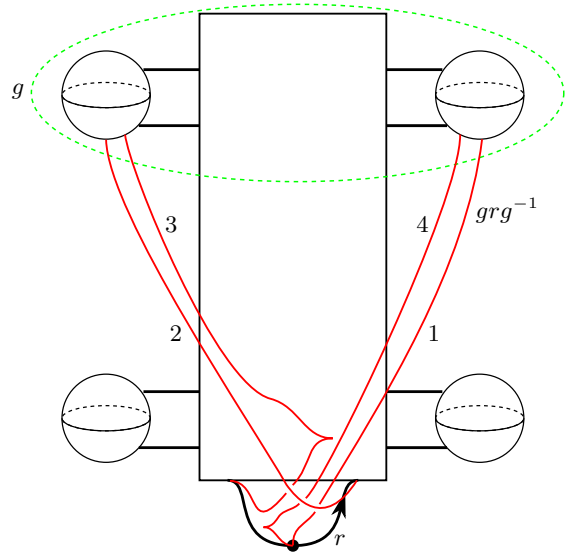}
  \caption{The Tietze move {\bf T 1} (ii).}
  \label{figure:tietze1ii}
\end{figure}

In Figure~\ref{figure:tietze1ii} we have indicated a base point
that one needs to fix in order to set up a one-to-one
correspondence between loops at this base point and words
in the generators. For the Tietze move, however, one
only needs to check that attaching a handle along
$r$ has the same effect as attaching it along $grg^{-1}$, and
for that the described Legendrian isotopy is not required to fix
the base point

\vspace{1mm}

{\bf T 1.} (iii)
This corresponds to a handle slide or second Kirby move in the sense
of~\cite{dige09}.

\vspace{1mm}

{\bf T 2.}
In Figure~\ref{fig:tietze2-1} we have indicated the relation $g=w$,
where the word $w$ is supposed to be given by a curve that may
go over the $1$-handles corresponding to the generators $g_1,\ldots ,g_k$.
We need to show that the generator $g$ and this relation
can be cancelled by diagram moves.

\begin{figure}[htp]
\labellist
\small\hair 2pt
\pinlabel $g$ [r] at 46 376
\pinlabel $w^{-1}$ at 254 212
\endlabellist
\centering
\includegraphics[scale=0.45]{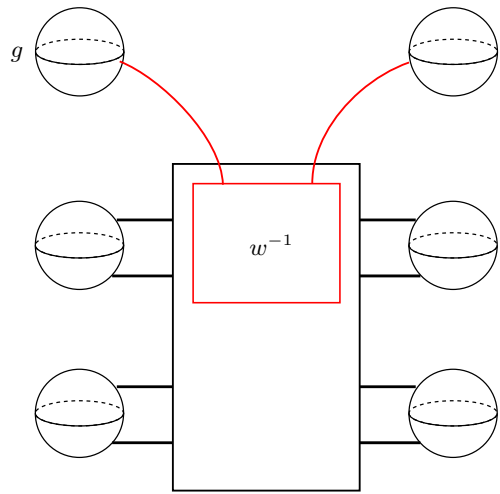}
  \caption{The relation $g=w$.}
  \label{fig:tietze2-1}
\end{figure}

Slide the right-hand attaching ball of the $1$-handle $g$ along the
Legendrian curve representing the word~$w^{-1}$; this can be done via
a contact isotopy. In the process, we ``accumulate'' cusps.
See Figure~\ref{fig:tietze2-2} for an example.

\begin{figure}[htp]
\centering
\includegraphics[scale=0.5]{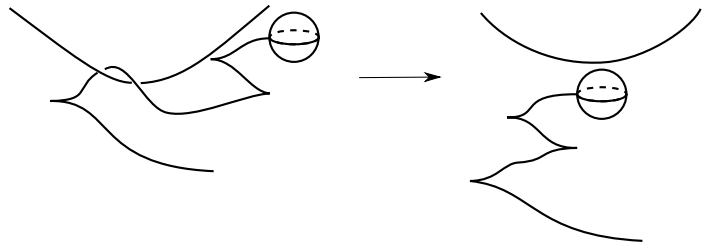}
  \caption{Sliding the attaching ball of a $1$-handle.}
  \label{fig:tietze2-2}
\end{figure}

Positive and negative Legendrian stabilisations can be removed in pairs
using the light bulb trick~\cite[Figure~21]{dige09}. So
ultimately we obtain a diagram as shown in
Figure~\ref{fig:tietze2-3}, cf.~\cite[Section~5]{dige10}.

\begin{figure}[htp]
\labellist
\small\hair 2pt
\pinlabel $g$ [r] at 46 384
\endlabellist
\centering
\includegraphics[scale=0.45]{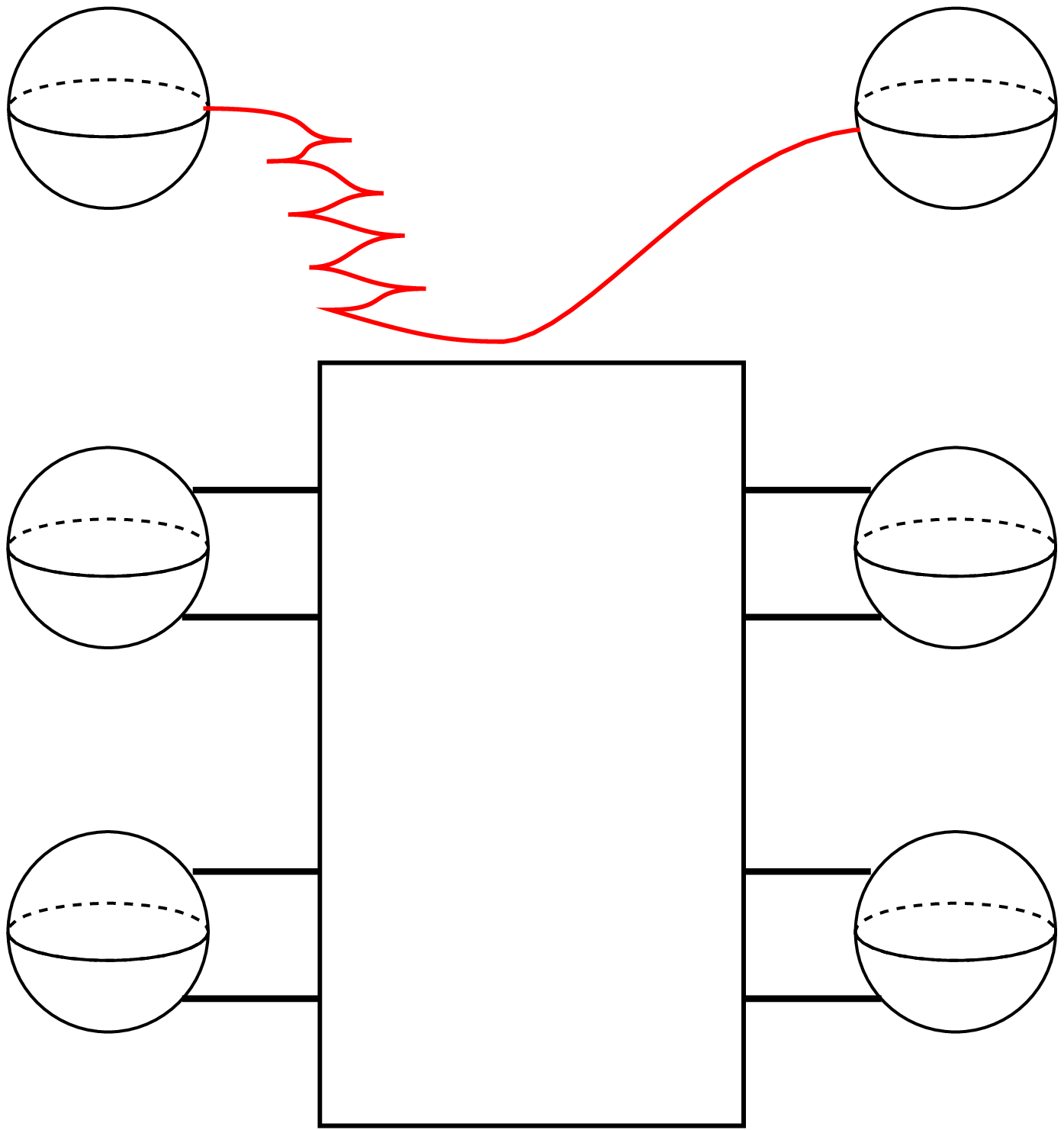}
  \caption{A cancelling handle pair.}
  \label{fig:tietze2-3}
\end{figure}

Then the $1$-handle $g$ and the $2$-handle attached along the
Legendrian circle going once over the $1$-handle form a cancelling
pair. {\em Topologically\/} this follows because the attaching
circle of the $2$-handle intersects the belt sphere of the $1$-handle
in exactly one point, see~\cite[Proposition~4.2.9]{gost99}.
In fact, this diagram with a cancelling handle pair also
represents the standard $4$-disc {\em symplectically\/},
see~\cite{vkoe}. In the given dimension, this is
also a consequence of the deep result of Gromov~\cite[p.~311]{grom85},
cf.~\cite[Theorem~9.4.2]{mcsa04},
that any strong symplectic filling of $(S^3,\xi_{\mathrm{st}})$
not containing homologically non-trivial $2$-spheres (in 
particular a filling known to be topologically a disc) is
{\em symplectomorphic} to the $4$-disc.
\subsection{Classification of subcritically fillable contact $5$-manifolds}
As a first step towards showing how the fundamental group
determines subcritically fillable contact $5$-manifolds,
the following lemma tells us that two diagrams representing isomorphic
fundamental groups can be turned into two diagrams
giving identical group presentations.

\begin{lem}
\label{lem:identical_presentations}
Let $(M,\xi)$ and $(M',\xi')$ be two subcritically Stein fillable contact
$5$-manifolds with $\pi_1(M)\cong \pi_1(M')$. Then there are
$k,k'\in\N_0$ and diagrams $D,D'$ for
\[ (M,\xi)\# k(S^2\times S^3,\xi_0)\;\;\;\mbox{and}\;\;\;
(M',\xi')\# k'(S^2\times S^3,\xi_0),\]
respectively, such that the presentations
of $\pi_1(M),\pi_1(M')$ determined by $D,D'$ are identical.
\end{lem}

\begin{proof}
Write $(M,\xi)=\open(\Sigma,\id)$, so that the subcritical Stein
filling is given by $\Sigma\times D^2$, and similarly for~$M'$.
The theorem of Seifert and van Kampen, applied to the decomposition
\[ M=\Sigma\times S^1\cup_{\partial}\partial\Sigma\times D^2, \]
shows that $\pi_1(M)\cong\pi_1(\Sigma)$. So any choice of Kirby diagram
for $\Sigma$ and $\Sigma'$ gives rise to two presentations of the same group.

Now invoke Tietze's theorem and our implementation of the Tietze moves.
Since we merely want to turn both group presentations into
identical ones by appropriate changes in the diagrams, rather than
converting one presentation to the other,
we need the Tietze move {\bf T~1} only going from left to right, i.e.\
we do not care about accumulating redundant relations.
As we saw in Section~\ref{section:realisingT}, the move {\bf T~1}~(i)
amounts to a connected sum with $(S^2\times S^3,\xi_0)$;
the other moves do not change the contact manifold.
\end{proof}

This lemma does not say anything about the contact framings of
the Legendrian knots in the diagrams $D,D'$. In particular,
we cannot expect that $M\# k(S^2\times S^3)$ and
$M'\# k'(S^2\times S^3)$ are diffeomorphic, in general.
The most simple example would be $M=S^2\times S^3$ and
$M'=S^2\,\tilde\times\,S^3$, where we could take $k=k'=0$.

The following lemma is the key to showing that a summand
$(S^2\,\tilde\times\,S^3,\xi_1)$ will give us complete control
over the contactomorphism type of the resulting manifold.

\begin{lem}
\label{lem:realise_legendrian_stabilisation}
Let $\Sigma$ be a Stein surface containing at least one $2$-handle
attached along a Legendrian knot~$K$. Let $\Sigma_{\pm}$ be
the Stein surface where this one attaching circle is
replaced by its Legendrian stabilisation $S_{\pm}K$. Then
\[ \open (\Sigma,\id)\# (S^2\,\tilde\times\,S^3,\xi_1)\cong
\open (\Sigma_{\pm},\id)\# (S^2\,\tilde\times\,S^3,\xi_1).\]
\end{lem}

\begin{proof}
The contact manifold $\open(\Sigma,\id)\#
(S^2\,\tilde\times\,S^3,\xi_1)$
is represented by a diagram for $\Sigma$ with one additional shark.
After a handle slide of $K$ over the shark (with one or the
other orientation), and by applying moves I and~II, we obtain
a diagram with $K$ replaced by $S_{\pm}K$, see
Figure~\ref{fig:stabilise_knot}. In (i) we have added a shark
to the given diagram, i.e.\ formed the connected sum with
$(S^2\,\tilde\times\,S^3,\xi_1)$. We then perform a handle slide
over the shark to obtain~(ii). From there one gets to (iii) and (iv)
as in Figure~\ref{fig:connected_sum_S2S3}.
\end{proof}

\begin{figure}[htp]
\labellist
\small\hair 2pt
\pinlabel {\small (i)} [r] at 25 251
\pinlabel {\small (ii)} [r] at 533 251
\pinlabel {\small (iii)} [r] at 25 86
\pinlabel {\small (iv)} [r] at 533 86
\endlabellist
\centering
\includegraphics[scale=0.4]{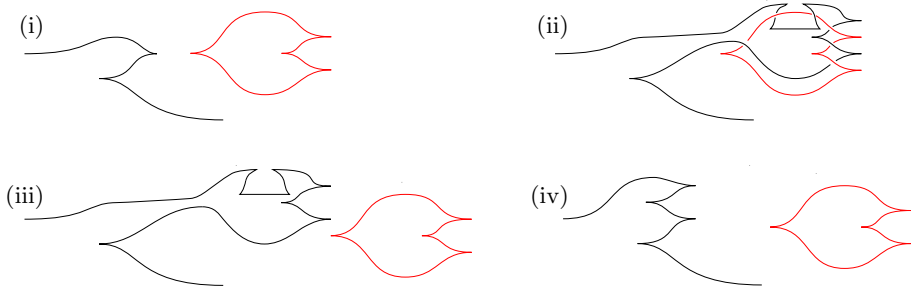}
  \caption{Handle slide turning $K$ into $S_{\pm}K$.}
  \label{fig:stabilise_knot}
\end{figure}

Here is our main classification result.

\begin{thm}
\label{thm:classification}
Let $(M,\xi)$ and $(M',\xi')$ be two subcritically Stein fillable contact
$5$-mani\-folds with $\pi_1(M)\cong \pi_1(M')$.
Then there are $k,k'\in\N_0$ with $k-k'=\rank H_2(M')-\rank H_2(M)$
such that
\[ (M,\xi)\,\#\, k(S^2\times S^3,\xi_0)\,\#\,
(S^2\,\tilde\times\,S^3,\xi_1)\cong
(M',\xi')\,\#\, k'(S^2\times S^3,\xi_0)\,\#\,
(S^2\,\tilde\times\,S^3,\xi_1).\]
\end{thm}

\begin{proof}
By Lemma~\ref{lem:identical_presentations} there are $k_0,k_0'\in\N_0$
such that $(M,\xi)\# k_0(S^2\times S^3,\xi_0)$ and
$(M',\xi')\# k'_0(S^2\times S^3,\xi_0)$ are represented by diagrams
that give identical presentations of $\pi_1(M)\cong\pi_1(M')$.
Notice that the two diagrams may contain Legendrian knots that do
not go over the $1$-handles and hence do not contribute
to the fundamental group or its presentation.

With move~II we can change the crossings of the Legendrian attaching
circles, so we may assume that the two diagrams are topologically
identical, up to a finite number of unlinked unknots. These unlinked unknots
correspond to summands $S^2\times S^3$ or $S^2\,\tilde\times\,S^3$ with
one of the standard contact structures described in
Proposition~\ref{prop:S2S3}.

By the theorem of Fuchs and Tabachnikov~\cite[Theorem~4.4]{futa97},
this topological isotopy can be
realised as a Legendrian isotopy of suitable stabilisations of the
Legendrian knots. As shown in the preceding lemma, a summand
$(S^2\,\tilde\times\,S^3,\xi_1)$ enables us to perform such
Legendrian stabilisations.

Finally, Proposition~\ref{prop:connect_sum_S2S3} allows us,
up to contactomorphism, to turn all but one summand
$(S^2\,\tilde\times\,S^3,\xi_1)$ into
$(S^2\times S^3,\xi_0)$.
\end{proof}

It is possible to combine this argument with that for
Theorem~\ref{thm:c_1_determines_5manifold} in order to formulate
a more precise classification result that involves conditions on
the first Chern class, but we have opted for the more
transparent formulation of the statement.

The moves in the proof of Theorem~\ref{thm:classification}
all extend to symplectomorphisms of the filling, so the
following corollary is immediate. Here we write $(S^2\times D^4,\omega_0)$
for the standard filling of $(S^2\times S^3,\xi_0)$,
and $(S^2\,\tilde\times\,D^4,\omega_1)$ for that of
$(S^2\,\tilde\times\,S^3,\xi_1)$.
`Symplectomorphism' is to be understood in the sense of
Remark~\ref{rem:extension}.

\begin{cor}
\label{cor:classification_Stein}
Let $W$ and $W'$ be two compact subcritical Stein manifolds of dimension $6$
with $\pi_1(W)\cong \pi_1(W')$. 
Then there are $k,k'\in\N_0$ with $k-k'=\rank H_2(W')-\rank H_2(W)$
such that the boundary connected sums
\[ W\,\natural\, k(S^2\times D^4,\omega_0)\,\natural\,
(S^2\,\tilde\times\,D^4,\omega_1)\]
and
\[ W'\,\natural\, k'(S^2\times D^4,\omega_0)\,\natural\,
(S^2\,\tilde\times\,D^4,\omega_1)\]
are symplectomorphic.\qed
\end{cor}

\begin{rem}
\label{rem:hprinciple}
It seems feasible to prove Theorem~\ref{thm:classification} by working
directly with the subcritical filling. The fundamental group
of the $6$-dimensional subcritical Stein manifold $\Sigma\times D^2$,
which equals that of~$M$, has a presentation with generators given by
the $1$-handles and relations given by the attaching circles of the
$2$-handles. In the $6$-dimensional Stein manifold we now have
an $h$-principle for the attaching maps.
\end{rem}
\section{Subcritical fillings of $S^5$}
\label{section:S5}
The purpose of the present section is to show how the
so-called Andrews--Curtis moves on balanced group
presentations can be realised as moves in our diagrams
for contact $5$-manifolds. As an application we prove
that any subcritically Stein fillable contact structure
on the $5$-sphere is the standard one, provided the
filling gives rise to an Andrews--Curtis trivial
presentation of the trivial group.

Let $\xi$ be a contact structure on $S^5$ admitting a subcritical
Stein filling~$W$. From the homology long exact sequence of
the pair $(W,S^5)$ we deduce that $W$ has the homology of
a point. By the result of Cieliebak~\cite{ciel02} we can write
$W=\Sigma\times D^2$. Hence $\pi_1(W)=\pi_1(\Sigma)=\pi_1(S^5)=\{ 1\}$,
cf.\ the proof of Lemma~\ref{lem:identical_presentations}.
From the $h$-cobordism theorem it follows that $W$ is diffeomorphic to~$D^6$.

Since the homology of $W$ can also be computed from a cellular
decomposition, it follows that a handlebody decomposition of $W$
with a single $0$-handle must contain an equal number
of $1$- and $2$-handles. This gives rise to a {\bf balanced
presentation} of the trivial group $\pi_1(W)$, i.e.\
a presentation containing as many relations as generators.

If $\partial\Sigma$ were the standard $3$-sphere, we could
conclude immediately that $\xi$ is the standard contact structure
on~$S^5$. Unfortunately, the boundary of a contractible $4$-manifold
$\Sigma$ (even with a finite handlebody decomposition as
described) will, in general, be some complicated homology
$3$-sphere, as pointed out in~\cite{gomp91}.
\subsection{Andrews--Curtis moves}
In the attempt to prove that $\xi$ is the standard contact structure
on~$S^5$, one can try to show that the balanced presentation
of the trivial group $\pi_1(W)$ can be converted to the empty
presentation via balanced presentations, and then to implement these
moves as transformations of the handlebody~$W$. As we shall see in the next
section, this topological implementation does not pose any difficulties.
The algebraic part of this strategy, however, remains unresolved
and forms the content of the Andrews--Curtis conjecture.
There are various forms of this conjecture; the following is the
weaker of the two versions given in~\cite{ancu65}.

\begin{conj}
Any balanced presentation
of the trivial group can be reduced to the empty presentation
(via balanced presentations) by the following transformations:
\begin{itemize}
\item[{\bf AC 1.}] Replace a relation by its inverse.
\item[{\bf AC 2.}] Conjugate a relation by a generator.
\item[{\bf AC 3.}] Replace one relation by its product with another
relation.
\item[{\bf AC 4.}] Add or remove a generator $g$ together with the
relation~$g$.
\end{itemize}
\end{conj}

A balanced presentation of the trivial group is called
{\bf Andrews--Curtis trivial} if it can be reduced to the empty
presentation using Andrews--Curtis moves.
In other words, the Andrews--Curtis conjecture can be rephrased as
saying that any balanced presentation of the trivial group is
Andrews--Curtis trivial.

This conjecture is still unresolved, but potential counterexamples
have been suggested in~\cite{gomp91}.
\subsection{Realising the Andrews--Curtis moves}
The move {\bf AC~1} amounts to reversing the orientation of
the Legendrian attaching circle representing the
relation in question;
the moves {\bf AC~2} and {\bf AC~3} are the
Tietze moves {\bf T~1}~(ii) and {\bf T~1}~(iii), respectively;
the fourth Andrews--Curtis move is a special case of the
second Tietze move. So all the Andrews--Curtis moves can be
realised in our diagrams as in Section~\ref{section:realisingT}.

Some versions of the Andrews--Curtis conjecture also allow a
generator to be replaced by its product with another
generator, or by its inverse. These moves, too, can be realised
in our diagrams as follows.

Multiplying a generator by another
one amounts to a $1$-handle slide~\cite[Figure~5.2]{gost99}.
Such a $1$-handle slide can also be performed in a diagram with
Legendrian attaching circles.

In order to replace a generator $g$ by its inverse, one needs to
flip the two attaching balls of the $1$-handle corresponding
to~$g$. This can be done after sliding the Legendrian curves
that go over this $1$-handle halfway around the attaching balls,
cf.~\cite[Figure~16]{gomp98}.
\subsection{A characterisation of the standard~$S^5$}
With the Andrews--Curtis moves at our disposal, we can now
give a characterisation of the standard contact structure
on $S^5$ in terms of subcritical Stein fillings.

\begin{prop}
Let $\xi$ be a contact structure on $S^5$ with a subcritical
Stein filling~$W$. If $W$ admits a plurisubharmonic Morse function that
induces a handlebody decomposition of $W$ giving rise to an
Andrews--Curtis trivial presentation of the trivial group $\pi_1(W)$,
then $\xi$ is diffeomorphic to the standard contact structure
on~$S^5$.
\end{prop}

\begin{proof}
The subcritical Stein manifold $W$ corresponds to a description of
$(S^5,\xi)$ as a contact open book $\open(\Sigma,\id)$ with a handlebody
decomposition of the page $\Sigma$ giving rise to an Andrews--Curtis
trivial presentation of the trivial group.
By realising the Andrews--Curtis moves that transform this
presentation to the empty one, we transform the open book
to the one described by the empty diagram, which represents
the standard contact structure on~$S^5$.
\end{proof}

Depending on one's predisposition, one may regard this
result as evidence for the conjecture that among the contact
structures on $S^5$ only the standard one admits a subcritical
Stein filling, or as a potential means for disproving the
Andrews--Curtis conjecture. Indeed, one way to read the
proposition is that any exotic (i.e.\ non-standard) contact
structure on $S^5$ with a subcritical Stein filling
$W$ gives rise to a balanced presentation of the trivial
group $\pi_1(W)$ that is not Andrews--Curtis trivial.
Unfortunately, the result of M.-L.\ Yau cited at the end
of Section~\ref{section:no1handles} implies that
cylindrical contact homology is not sensitive enough to detect
such examples.

Although the general belief appears to be that the Andrews--Curtis conjecture
is false, we should admit in all fairness that this suggested
strategy for disproving it is not the most promising one.
It seems more likely that methods such as those indicated in
Remark~\ref{rem:hprinciple} allow one
to give a direct proof that the standard contact structure
on $S^5$ is the only one admitting a subcritical Stein filling.
\section{Diagrams for simply connected $5$-manifolds}
\label{section:simply}
In this section we exhibit diagrams for contact structures on
some simply connected $5$-manifolds. As a corollary, we obtain
a branched covering  description of $5$-dimensional simply connected
spin manifolds, and a new proof that every simply connected $5$-manifold
admits a contact structure in each homotopy class of
almost contact structures.
\subsection{Barden's classification}
Barden~\cite{bard65} has given a complete classification of
simply connected $5$-manifolds. Here we are only interested in those
manifolds that potentially carry contact structures.
If a $5$-manifold $M$ admits a
contact structure, then its structure group reduces to
$\UU (2)\times 1$; such a reduction is called
an {\bf almost contact structure}. Necessary and sufficient for the
existence of an almost contact structure is the vanishing of
the third integral Stiefel--Whitney class~$W_3(M)$,
see~\cite[Proposition~8.1.1]{geig08}.

According to Barden's classification, any simply connected $5$-manifold
$M$ with $W_3(M)=0$ decomposes as the
connected sum of finitely many manifolds from the
following list of examples:

\begin{itemize}
\item[(i)] manifolds $M_k$, $k\in\N$, characterised by
$H_2(M_k)\cong\Z_k\oplus\Z_k$,
\item[(ii)] $S^2\times S^3$,
\item[(iii)] $S^2\,\tilde\times\,S^3$.
\end{itemize}

The manifold $M_1$ is the $5$-sphere; the manifolds $M_k$ are prime
if and only if $k$ is a prime power $p^j$, $j\geq 1$. One
obtains a unique prime decomposition of a given $M$
if one requires that only the $M_{p^j}$ and at most one
summand $S^2\,\tilde\times\,S^3$ are used,
cf.\ Proposition~\ref{prop:connect_sum_S2S3}.
All the prime manifolds in this list, with the exception of
$S^2\,\tilde\times\,S^3$, are spin manifolds, i.e.\ their second
Stiefel--Whitney class vanishes.
\subsection{Diagrams for the $M_k$}
\label{section:Mk}
Let $(N_k,\eta_k)$, $k\in\N$, be the contact $5$-manifold represented
by the diagram depicted in Figure~\ref{fig:diagramM_k}. Write
$\Sigma_k$ for the page of the open book represented
by this diagram, so that $(N_k,\eta_k)=
\open (\Sigma_k,(\tau_{K_1}\circ \tau_{K_2})^2)$.

\begin{figure}[htp]
\labellist
\small\hair 2pt
\pinlabel $K_1$ [br] at 46 163
\pinlabel $K_2$ [tr] at 46 71
\pinlabel {\scriptsize $k$ twists} at 152 140
\pinlabel $(\tau_{K_1}\circ\tau_{K_2})^2$ [l] at 255 60
\endlabellist
\centering
\includegraphics[scale=0.7]{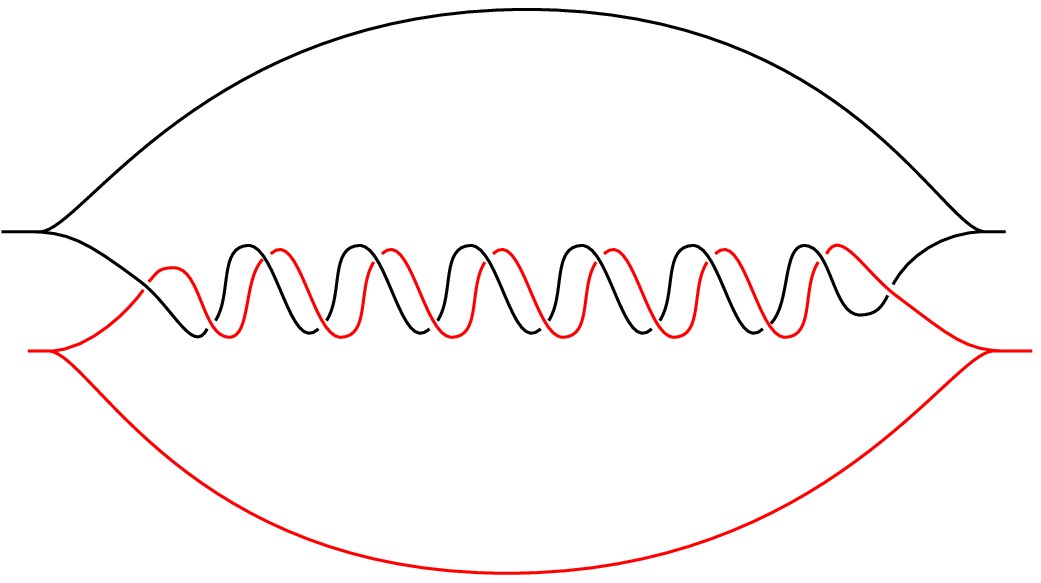}
  \caption{Diagram for $M_k$.}
  \label{fig:diagramM_k}
\end{figure}

\begin{prop}
\label{prop:M_kN_k}
For each $k\in\N$ the manifold $N_k$ is diffeomorphic to $M_k$.
\end{prop}

One ingredient in the argument will be the following folklore
theorem, which is proved in~\cite{vkoe}; cf.\
\cite{geig} for a proof in the $3$-dimensional case.

\begin{thm}
\label{thm:Dehntwist_legendrian_surgery}
Let $L\subset\open(\Sigma,\psi)$ be a Legendrian
sphere in a contact open book that sits on a page of the
open book as a Lagrangian submanifold.
Then the contact manifold obtained by Legendrian surgery along
$L$ is contactomorphic to $\open(\Sigma,\psi\circ\tau_L)$,
where $\tau_L$ denotes the right-handed Dehn twist
along $L\subset\Sigma$.\qed
\end{thm}

\begin{proof}[Proof of Proposition~\ref{prop:M_kN_k}]
We construct a Stein filling $W_k$ of $(N_k,\eta_k)$ by
starting with the subcritical Stein manifold $\Sigma_k\times D^2$,
which is a filling for $\open (\Sigma_k,\id)$, and attaching a $3$-handle
for each Dehn twist. The first two $3$-handles corresponding to
$\tau_{K_1}$ and $\tau_{K_2}$ cancel the $2$-handles of~$\Sigma_k$,
so we obtain a $6$-dimensional disc $D^6$.

The further two $3$-handles corresponding to the iterated application
of $\tau_{K_1}$ and $\tau_{K_2}$ then yields the Stein filling~$W_k$.
This implies $H_3(W_k)\cong\Z^2$, and the skew-symmetric
intersection form $Q_{W_k}$ on $H_3(W_k)$ has to look like
\[ \left(\begin{array}{cc}
0    & l_k \\
-l_k & 0   \\
\end{array}\right) \]
for some $l_k\in\Z$. By the argument mentioned in
Section~\ref{section:homology_openbook} we conclude
\[ H_2(N_k)\cong H_2(\partial W_k)\cong\coker Q_{W_k}\cong
\Z_{l_k}\oplus\Z_{l_k}.\]

The manifold $N_k$ is obviously simply connected, and it satisfies
$W_3(N_k)=0$ because it carries a contact structure. So in order to
establish that $N_k$ is diffeomorphic to $M_k$ it suffices to
show, by Barden's classification, that $|l_k|=k$.

Decompose the open book $N_k$ as $N_k=A_k\cup_{\partial} B_k$ as
in Section~\ref{section:openbooks_defs}, with $A_k$ the mapping torus
of~$\Sigma_k$, and $B_k=\partial\Sigma_k\times D^2\simeq\partial\Sigma_k$.
We can then compute the homology of $N_k$
by the procedure outlined in Section~\ref{section:homology_openbook}.
In particular, we consider the Mayer--Vietoris sequence
\begin{multline*}
H_3(N_k)\longrightarrow H_2(A_k\cap B_k)\longrightarrow
H_2(A_k)\oplus H_2(B_k)\longrightarrow \\
H_2(N_k)\longrightarrow
H_1(A_k\cap B_k)\longrightarrow H_1(A_k)\oplus H_1(B_k).
\end{multline*}

In terms of the standard generators
of $H_2(\Sigma_k)$, the intersection form $Q_{\Sigma_k}$
is given by
\[ \left(\begin{array}{cc}
-2 & k  \\
k  & -2 \\
\end{array}\right). \]
Thus, as shown in Section~\ref{section:action}, the action of
the two Dehn twists on $H_2(\Sigma_k)$ is given by
\[ (\tau_{K_1})_*=
\left(\begin{array}{cc}
-1 & k \\
0  & 1
\end{array}\right)
\;\;\;
\mbox{\rm and}
\;\;\;
(\tau_{K_2})_*=
\left(\begin{array}{cc}
1 & 0  \\
k & -1
\end{array}\right),\]
hence
\[ (\tau_{K_1}\circ\tau_{K_2})^2_*=
\left(\begin{array}{cc}
k^4-3k^2+1 & 2k-k^3 \\
-2k+k^3    & -k^2+1
\end{array}\right). \]
From the Wang sequence of the mapping torus $A_k$ we then have
\[ H_2(A_k)\cong\coker
\left(\begin{array}{cc}
k^4-3k^2 & 2k-k^3 \\
-2k+k^3    & -k^2
\end{array}\right). \]

Given a homomorphism $\phi\co\Z^m\rightarrow\Z^m$ with $\det\phi\neq 0$,
a simple algebraic consideration shows that $\lvert\coker\phi\rvert=
\lvert\det\phi\rvert$.
This allows us to conclude that
\[ |H_2(A_k)|=k^2|k^2-4|\;\;\mbox{\rm for}\;\; k\neq 2.\]
For $k=2$ we obtain $H_2(A_2)\cong \Z\oplus\Z_4$.

From the K\"unneth theorem we have
\[ H_1(A_k\cap B_k)\cong H_0(\partial\Sigma_k)\oplus H_1(\partial\Sigma_k).\]
From the Wang sequence one sees that $H_1(A_k)$ is
generated by the class of $\{ p\}\times S^1$, where $p$ is
any point of~$\partial\Sigma_k$. Combining these two
pieces of information, one deduces that the homomorphism
\[ H_1(A_k\cap B_k)\longrightarrow H_1(A_k)\oplus H_1(B_k)\]
in the Mayer--Vietoris sequence is an isomorphism.

Similarly, we have
\[ H_2(A_k\cap B_k)\cong H_1(\partial\Sigma_k)\oplus
H_2(\partial\Sigma_k),\]
and the second summand $H_2(\partial\Sigma_k)$
maps isomorphically into the
second summand of $H_2(A_k)\oplus H_2(B_k)$ in the
Mayer--Vietoris sequence.

So that sequence tells us that $H_2(N_k)$ is a quotient of
$H_2(A_k)$, and hence at most of rank~$1$. Combining this
with the information that $H_2(N_k)$ is isomorphic to
$\Z_{l_k}\oplus\Z_{l_k}$, we conclude $l_k\neq 0$.
Since $H_1(N_k)=0$, we obtain $H^2(N_k)=0$ from the
universal coefficient theorem, whence $H_3(N_k)=0$ by
Poincar\'e duality. Thus, the Mayer--Vietoris sequence
reduces to
\[ 0\longrightarrow H_1(\partial\Sigma_k)\longrightarrow
H_2(A_k)\longrightarrow H_2(N_k)\longrightarrow 0.\]

Recall that $H_1(\partial\Sigma_k)\cong\coker Q_{\Sigma_k}$.
For $k\neq 2$ this is a finite group of order $|k^2-4|$.
It follows that $H_2(N_k)$ is a finite group of order $k^2$,
and hence $|l_k|=k$, as we wanted to show.

For $k=2$ the short exact sequence becomes
\[ 0\longrightarrow \Z\oplus\Z_2\longrightarrow
\Z\oplus\Z_4\longrightarrow H_2(N_k)\longrightarrow 0.\]
Recall that $H_1(\partial\Sigma_2)$ is generated by
meridional loops $u_1,u_2$ around $K_1,K_2$,
respectively. The homomorphism $H_1(\partial\Sigma_2)\rightarrow
H_2(A_2)$ is given by sending $u_i$ to the class represented by
the torus $u_i\times S^1\subset\partial\Sigma_2\times S^1\subset A_2$.
Now cut this torus along a meridian $u_i\times *$ and insert two
meridional discs in~$\Sigma_2$ (of opposite orientation); this gives
us a $2$-sphere in $A_2$ representing the same homology class. Now flow
one of the meridional discs and the cylindrical part of that
$2$-sphere along a vector field on the mapping torus that is transverse
to the fibres and whose time-1 map, say, gives the monodromy map
from a fibre to itself.

The geometric intersection number of the meridional disc $D_i$ to $u_i$
with the spherical generator $S_j$ of $H_2(\Sigma_2)$ corresponding to
$K_j$ (made up of a Seifert disc for $K_j$ and the core disc of the
$2$-handle) equals the Kronecker~$\delta_{ij}$. The self-intersection number
of the $S_i$ is $-2$, and the intersection number between $S_1$ and
$S_2$ equals $k=2$. Hence, with the observation from
Section~\ref{section:action} we infer that the monodromy
$(\tau_{K_1}\circ\tau_{K_2})^2$ acts on the $D_i$ as follows:
\[ \begin{array}{ccccc}
D_1 &\longmapsto & D_1           &\longmapsto & D_1+S_1        \\
    &\longmapsto & D_1+S_1+2S_2  &\longmapsto & D_1+4S_1+2S_2, \\
D_2 &\longmapsto & D_2+S_2       &\longmapsto & D_2+2S_1+S_2   \\
    &\longmapsto & D_2+2S_1+4S_2 &\longmapsto & D_2 +6S_1+4S_2.
\end{array} \]
The homology group $H_2(A_2)\cong\Z\oplus\Z_4$ is generated by the classes
of $S_1$ and $S_2$, subject to the relation $4(S_1+S_2)=0$.
So the effect of the monodromy homomorphism can also be written as
\[ \begin{array}{ccc}
D_1 & \longmapsto & D_1-2S_2,\\
D_2 & \longmapsto & D_2+2S_1.
\end{array} \]
It follows that the $2$-torus $u_1\times S^1$ (resp.\ $u_2\times S^1$)
is homologically equivalent to $-2S_2$ (resp.\ $2S_1$) in~$A_2$.
Thus, in terms of suitable generators,
the homomorphism $\Z\oplus\Z_2\rightarrow\Z\oplus\Z_4$
in the Mayer--Vietoris sequence for $N_2$ is multiplication by~$2$,
hence $H_2(N_2)\cong\Z_2\oplus\Z_2$, i.e.\ $|l_2|=2$.
\end{proof}
\subsection{Spin manifolds as branched covers}
Here is an amusing corollary of the above example.

\begin{cor}
Any closed, simply connected $5$-dimensional spin manifold is a double
branched cover of the $5$-sphere.
\end{cor}

\begin{proof}
By Barden's classification, any closed, simply connected
$5$-dimensional spin manifold $M$ can be decomposed as
\[ M\cong \#_mS^2\times S^3\# M_{k_1}\#\cdots\# M_{k_n}.\]
This classification of {\em spin\/} manifolds had actually been
achieved earlier by Smale~\cite{smal62}.
By the preceding section, each $M_{k_i}$ is diffeomorphic
to an open book $\open (\Sigma_{k_i},\psi_i^2)$.
The manifold $S^2\times S^3$ can likewise be written as an open book
with quadratic monodromy. Namely, as page take
the cotangent unit disc bundle $DT^*S^2$
of~$S^2$. With Proposition~\ref{prop:S2S3} the
manifold $\open (DT^*S^2,\id)$
can be shown to be diffeomorphic to $S^2\times S^3$.
Let $\tau$ be a right-handed
Dehn twist along the zero section of~$DT^*S^2$. Since $\tau^2$
is isotopic, relative to the boundary, to the identity map,
it follows that $S^2\times S^3$ is diffeomorphic to
$\open (DT^*S^2,\tau^2)$, see \cite{vkni05}
and~\cite[p.~36]{hima68}. So $M$ can be written as
\[ M\cong\open (\natural_mDT^*S^2\natural\Sigma_{k_1}\natural
\cdots\natural\Sigma_{k_n},\natural_m\tau^2\natural
\psi_1^2\natural\cdots\natural\psi_n^2).\]
This implies that $M$ is the double branched cover of
\[ \open (\natural_mDT^*S^2\natural\Sigma_{k_1}\natural
\cdots\natural\Sigma_{k_n},\natural_m\tau\natural
\psi_1\natural\cdots\natural\psi_n),\]
branched along the binding of the open book. That last open book,
however, is a right-handed stabilisation of
$\open (D^4,\id)$, which is diffeomorphic to the $5$-sphere.
\end{proof}
\subsection{Existence of contact structures}
The following theorem was first proved by the second author~\cite{geig91},
using contact surgery. The second proof was given by the third
author~\cite{vkoe08}, using open book decompositions.
Here we give a diagrammatic proof.

\begin{thm}
Every closed, oriented, simply connected $5$-manifold admits a contact
structure in each homotopy class of almost contact structures.
\end{thm}

\begin{proof}
Homotopy classes of almost contact structures on oriented, simply connected
$5$-manifolds are classified by the first Chern class,
cf.~\cite[Proposition~8.1.1]{geig08}. So each $M_k$ in
Barden's classification admits a unique almost contact structure
up to homotopy (for either of its orientations). The result now follows
from Section~\ref{section:Mk} and Proposition~\ref{prop:S2S3}.
\end{proof}
\begin{ack}
F.~D.\ is partially supported by
grant no.\ 10631060 of the National Natural Science Foundation
of China.

This project was initiated during the conference
``Symplectic Topology, Contact Topology, and Applications''
at Hokkaido University, organised by O.~v.~K.\ (then a post-doc
at Hokkaido University) and funded by JSPS grant 19$\cdot$07804.

Some parts of this research were done while
H.~G.\ and O.~v.~K.\ attended the workshop ``Symplectic Techniques in
Conservative Dynamics'' at the Lorentz Center, Leiden, organised by
V.~Ginzburg, F.~Pasquotto, B.~Rink and R.~Vandervorst. We thank the
organisers for this opportunity to meet, and the Lorentz Center for
the excellent working environment.

The final writing was done while F.~D.\ stayed at the Universit\"at zu
K\"oln, supported by a DAAD -- K.~C.~Wong fellowship, grant no.\
A/09/99005.
\end{ack}

\end{document}